\documentclass[a4paper,10pt]{article}

\usepackage[utf8]{inputenc}
\usepackage[width=15cm,height=22cm]{geometry}
\usepackage{amsfonts,amsmath,amssymb,amsthm,mathrsfs}
\usepackage[dvipsnames]{xcolor}
\usepackage[hidelinks,pdfusetitle]{hyperref}
\usepackage[capitalise]{cleveref}

\newtheorem{proposition}{Proposition}[section]
\newtheorem{theorem}[proposition]{Theorem}
\newtheorem{definition}[proposition]{Definition}
\newtheorem{corollary}[proposition]{Corollary}

\newtheorem{lemma}[proposition]{Lemma}
\newtheorem{remark}[proposition]{Remark}
\newtheorem{example}[proposition]{Example}
\numberwithin{equation}{section}

\title{Local convergence of Sussmann's infinite product \\ for analytic vector fields}
\author{Jérémy Le Borgne\thanks{Univ Rennes, CNRS, IRMAR - UMR 6625, F-35000 Rennes, France}, Frédéric Marbach\thanks{DMA, École normale supérieure, Université PSL, CNRS, 75005 Paris, France}}

\newcommand{\N}{\mathbb{N}}
\newcommand{\R}{\mathbb{R}}
\newcommand{\B}{\mathcal{B}}

\newcommand{\ad}{\operatorname{ad}}
\newcommand{\Br}{\operatorname{Br}}
\newcommand{\eval}{\textnormal{\textsc{e}}}
\newcommand{\Lyn}{\mathrm{Lyn}(X)}
\newcommand{\dd}{\,\mathrm{d}}
\newcommand{\norm}[1]{\lVert#1\rVert}

\begin{document}

\maketitle

\begin{abstract}
    In 1986, Sussmann proposed an expansion of the Chen series as an infinite product of exponentials of Hall basis elements multiplied by coefficients defined through a simple induction.
    He proved that this product converges locally in time for bilinear control systems associated with matrices or bounded operators.
    
    In this note, we prove that Sussmann's product also converges locally in time for control-affine systems driven by analytic vector fields.
    The difficulty is that one must keep track of the structure of the involved Lie brackets, as the natural decay of the explicit coefficients is not enough to counterbalance the natural growth of the iterated Lie brackets.
\end{abstract}

\section{Introduction}

\subsection{Sussmann's infinite product}

Sussmann's infinite product is a factorization of the Chen series into an ordered product of exponentials indexed by a Hall basis. 
Its interest is twofold: it preserves the Lie structure of the dynamics and its coefficients are given by an explicit induction. 
We recall the construction briefly, referring to \cite[Section~2]{BeauchardLeBorgneMarbach2023} for the basic definitions and more details.

Let $I$ be a finite set, and let $X=\{X_i \mid i \in I\}$ be a finite alphabet of non-commuting indeterminates. 
We denote by $\mathcal{A}(X)$ the free associative algebra of non-commutative polynomials in $X$, and by $\widehat{\mathcal{A}}(X)$ its completion, consisting of formal non-commutative power series. 
Similarly, $\mathcal{L}(X)$ denotes the free Lie algebra generated by $X$, and $\widehat{\mathcal{L}}(X)$ its completed counterpart.

We consider the following differential equation on $\widehat{\mathcal{A}}(X)$ (see e.g.\ \cite[Section~3]{Sussmann1983}):
\begin{equation}
    \label{eq:S}
    \dot{S}(t) = S(t) \left( \sum_{i \in I} u_i(t) X_i \right)
    \quad \text{and} \quad
    S(0) = 1,
\end{equation}
where $u = (u_i)_{i \in I} \in L^1((0,T);\R^I)$ is the input of the system.

The most direct representation of its solution is the Chen series \cite{Chen1954,Chen1957}, obtained by repeated integration of \eqref{eq:S}:
\begin{equation}
    \label{eq:Chen}
    S(t) = \sum_{\omega \in I^*} \left(\int_0^t u_\omega\right) X_\omega
\end{equation}
where $I^*$ is the free monoid on $I$, and, for any $\omega = (i_1, \dots, i_k) \in I^*$, $X_\omega := X_{i_1} \dotsb X_{i_k}$ and 
\begin{equation}
    \label{eq:Chen-coeff}
    \int_0^t u_\omega := \int_{0 < \tau_1 < \dotsb < \tau_k < t} u_{i_1}(\tau_1) \dotsb u_{i_k}(\tau_k),
\end{equation}
with the convention that $X_\emptyset = 1$ and $\int_0^t u_\emptyset = 1$.
In rough paths theory, this object is known as the \emph{signature} of the bounded-variation path $U(t) := (\int_0^t u_i)_{i \in I}$ (see e.g.\ \cite[Ch.\ 2]{LyonsCaruanaLevy2007} or \cite[Exercise 2.1]{FrizHairer2020}).
An advantage of \eqref{eq:Chen} is that the coefficients \eqref{eq:Chen-coeff} are easy to compute.
Its drawback is that it is written in the associative algebra and therefore does not display the Lie structure of the equation.

The Magnus expansion \cite{Magnus1954} gives a complementary point of view. 
Given a basis $\B$ of $\mathcal{L}(X)$, one can prove that there exist coefficients $\zeta_b(t,u) \in \R$ such that
\begin{equation}
    \label{eq:Magnus}
    S(t) = \exp \left( \sum_{b \in \B} \zeta_b(t,u) b \right).
\end{equation}
This representation is Lie-theoretic and much sparser than the Chen series. 
However, one does not know easy formulas for the \emph{coordinates of the first kind} $\zeta_b$.

A classical alternative, introduced by Wei and Norman \cite{WeiNorman1964} in the finite-dimensional setting, is to choose an order on $\B$ and seek a factorization
\begin{equation}
    \label{eq:sussmann-formal}
    S(t)
    =
    \overset{\longleftarrow}{\prod_{b \in \B}}
    \exp(\xi_b(t,u)b),
\end{equation}
where the $\xi_b(t,u) \in \R$ are called \emph{coordinates of the second kind}.
In finite dimension this is a local product of finitely many exponentials.
In the free Lie algebra, the product is infinite. 
Sussmann proved in \cite{Sussmann1986} that, when $\B$ is a particular type of basis of $\mathcal{L}(X)$ called a Hall set (see \cref{def:Hall}), this infinite product is meaningful at the formal level and the coefficients $\xi_b$ are given by a very nice explicit inductive formula (see \cref{def:xi}).

We refer to \cite{BeauchardLeBorgneMarbach2023,Kawski2000,Kawski2011} for more detailed comparisons of the above expansions.

\subsection{Convergence for real-analytic vector fields}

We now pass from formal series to control systems driven by vector fields. 
Let $p \in \R^d$, and let $(f_i)_{i \in I}$ be real-analytic vector fields defined in a neighborhood of $0 \in \R^d$. 
We consider the control-affine system
\begin{equation}
    \label{eq:syst}
    \dot{x}(t) = \sum_{i \in I} u_i(t) f_i(x)
    \quad \text{and} \quad
    x(0) = p,
\end{equation}
where $u = (u_i)_{i \in I} \in L^1((0,T);\R^I)$ is a control. This class includes systems with drift by singling out an index $0 \in I$ and setting $u_0 \equiv 1$. 
Whenever it is well-defined, the unique absolutely continuous solution to \eqref{eq:syst} is denoted by $x(t;u,p)$.

Representation formulas for the nonlinear system \eqref{eq:syst} can be derived from formal formulas for the linear abstract system \eqref{eq:S}; see \cite[Section~4.1]{BeauchardLeBorgneMarbach2023} for a short tutorial.
This technique was introduced in \cite{AgrachevGamkrelidze1979} and is also used for the design of numerical splitting methods (see e.g.\ \cite{CasasChartierMurua2019}).
However, this passage involves convergence issues, which are sometimes delicate.
The purpose of this note is to prove that Sussmann's product does converge locally for real-analytic vector fields.
This gives a positive answer to the question raised in \cite[Open Problem~135]{BeauchardLeBorgneMarbach2023}.

\begin{theorem}
    \label{thm:main}
    Let $I$ be a finite non-empty set.
    Let $\Omega \subset \R^d$ be an open neighborhood of $0$ and $(f_i)_{i \in I} \in (C^\omega(\Omega;\R^d))^I$.
    Let $\B$ be a Hall set on $X = \{ X_i \mid i \in I \}$.
    There exists $\eta > 0$ such that, for any $T > 0$, $u \in L^1((0,T);\R^I)$ and $p \in \R^d$ with $|p| \le \eta$ and $\norm{u}_{L^1} \le \eta$, and for any $t \in [0,T]$,
    \begin{equation}
        \label{eq:thm.main}
        x(t;u,p) = \left(\overset{\longrightarrow}{\prod_{b \in \B}} e^{\xi_b(t,u) f_b}\right)(p),
    \end{equation}
    where $\xi_b(t,u) \in \R$ is a coefficient defined in \cref{def:xi}, $f_b \in C^\omega(\Omega;\R^d)$ is an iterated Lie bracket defined in \cref{def:f_b}, and $e^{\xi_b(t,u) f_b}$ denotes the time-one flow of the autonomous vector field $\xi_b(t,u) f_b$ (which is well-defined under the given smallness assumption).
    The infinite product is defined as a net over finite subsets and converges uniformly on a small ball (see \cref{sec:proof-def}).
\end{theorem}

\begin{remark}
    The directions of the ordered products in \eqref{eq:sussmann-formal} and \eqref{eq:thm.main} are opposite because their is an inversion in the passage from formal series and linear operators to compositions of flows of nonlinear vector fields.
\end{remark}

\begin{example}
    Take $I=\{0,1\}$ and $d=2$.
    Consider the following real-analytic vector fields
    \begin{equation}
        f_0(x):=\left(0,\frac{x_1}{1-x_1}\right)
        \quad \text{and} \quad
        f_1(x):=(1,0).
    \end{equation}
    To illustrate the computations, take controls $u_0(t)=1$ and $u_1(t)=2t$. 
    Then
    \begin{equation}
        \dot x_1(t)=2t,
        \qquad
        \dot x_2(t)=\frac{x_1(t)}{1-x_1(t)}.
    \end{equation}
    Starting from $p=(0,0)$, explicit integration gives, for $0\leq t<1$,
    \begin{equation}
        \label{eq:explicit-ode-ex-t}
        x_1(t)=t^2,
        \qquad
        x_2(t)
        =
        \int_0^t\frac{s^2}{1-s^2}\dd s
        =
        \sum_{k\geq1}\frac{t^{2k+1}}{2k+1}
        = \operatorname{artanh}(t) - t.
    \end{equation}

    Let $\B$ be a length-compatible Hall set with $X_1<X_0$ (as in \cite{Hall1950,Hall1933}, see \cite[Section 6]{BeauchardLeBorgneMarbach2022}), and set
    \begin{equation}
        h_k:=\ad_{X_1}^k(X_0),
        \qquad k\geq1.
    \end{equation}
    Then $h_k\in\B$ for all $k\geq1$, and $X_1 < X_0 < h_1 < h_2 < \dotsb$ up to possible interleaving of Hall brackets of the same length. 
    The relevant Sussmann coefficients (see \cref{def:xi}) are
    \begin{equation}
        \xi_{X_0}(t)=t,
        \qquad
        \xi_{X_1}(t)=t^2,
        \qquad
        \xi_{h_k}(t)
        =
        \frac1{k!}\int_0^t s^{2k}\dd s
        =
        \frac{t^{2k+1}}{(2k+1)k!}.
    \end{equation}
    On the vector-field side, one has
    \begin{equation}
        f_0 = \frac{x_1}{1-x_1} \partial_{x_2}, \qquad
        f_1 = \partial_{x_1}, \qquad 
        f_{h_k}
        = \ad_{f_1}^k(f_0)
        = \frac{k!}{(1-x_1)^{k+1}} \partial_{x_2},
        \qquad k\geq1.
    \end{equation}
    All other Hall brackets give zero vector fields: the vertical fields
    $f_0,f_{h_1},f_{h_2},\ldots$ commute with each other, and bracketing with
    $f_1$ only differentiates in $x_1$.

    Hence, after removing identity factors, the truncation of Sussmann's product at length $N+1$ is
    \begin{equation}
        G_N(t)
        =
        e^{t^2f_1}
        \circ e^{t f_0}
        \circ e^{\xi_{h_1}(t)f_{h_1}}
        \circ\cdots\circ
        e^{\xi_{h_N}(t)f_{h_N}}.
    \end{equation}
    Since the factors $e^{\xi_{h_k}(t)f_{h_k}}$ are applied at $x_1=0$ and
    $f_{h_k}(0)=k!\partial_{x_2}$, while $f_0(0)=0$, we obtain
    \begin{equation}
        G_N(t)(0,0)
        =
        \left(
        t^2,
        \sum_{k=1}^N \frac{t^{2k+1}}{2k+1}
        \right).
    \end{equation}
    Therefore, for $0\leq t<1$,
    \begin{equation}
        \left(
        \overset{\longrightarrow}{\prod_{b\in\B}}
        e^{\xi_b(t,u) f_b}
        \right)(0,0)
        =
        \left(
        t^2,
        \operatorname{artanh}(t) - t
        \right),
    \end{equation}
    which equals \eqref{eq:explicit-ode-ex-t}.
\end{example}

\subsection{Known results and main difficulty}

The elements $b$ of a Hall set $\B$ have a length $|b| \in \N^*$; see \cref{sec:Hall}. For each $n \ge 1$, there are at most $|X|^n$ elements of length $n$.

Sussmann proved \cref{thm:main} in the linear case, namely when $f_i(x)=A_i x$ for matrices $A_i \in \mathcal{M}_d(\R)$, and more generally for bounded operators on a Banach space \cite{Sussmann1986}. 
In that setting the convergence mechanism is straightforward.
If $M_A := \max_{i \in I}\norm{A_i}$, then a direct induction gives
\begin{equation}
    \norm{A_b} \le (2M_A)^{|b|}.
\end{equation}
Moreover, the inductive definition of $\xi_b$ (see \cref{def:xi}) yields
\begin{equation}
    |\xi_b(t,u)| \le \norm{u}_{L^1}^{|b|}.
\end{equation}
Together with the bound on the number of brackets of a fixed length, these estimates imply that
\begin{equation}
    \sum_{b \in \B} \norm{\xi_b(t,u)A_b} < +\infty
\end{equation}
for small enough $\norm{u}_{L^1}$, which entails the convergence of the infinite product (see \cite[Section 5.4.2]{BeauchardLeBorgneMarbach2023}).

This smallness assumption cannot be removed. Sussmann constructed a counterexample, recalled in \cite[Proposition~132]{BeauchardLeBorgneMarbach2023}, showing that convergence may fail for large controls or large times.
This counterexample relies on a Hall set within which there exists a sequence $b_k$ such that $|b_k| \to \infty$ and $\xi_{b_k}(t,u) \ge c_0 (t/t_0)^{|b_k|}$ and matrices such that $\norm{A_{b_k}} \ge c_1^{|b_k|}$.
Hence the products $\xi_b(t,u) A_b$ do not even tend to $0$.

Such an example might hint towards a non-convergence result for real-analytic vector fields (even under appropriate smallness assumptions).
Indeed, for nonlinear vector fields, the associated differential operators are unbounded, so one cannot even hope for a geometric bound for $f_b$.
For example, it is known that the Magnus expansion \eqref{eq:Magnus} does not converge, even locally, in very weak senses, for real-analytic vector fields (see \cite[Section 5.2.3]{BeauchardLeBorgneMarbach2023}).
As we mentioned in \cite[Section 5.4.3]{BeauchardLeBorgneMarbach2023}, the best bound one can hope based solely on the length of the bracket is of the form 
\begin{equation}
    \norm{f_b} \le C^{|b|} |b|!.
\end{equation}

Thus the proof of \cref{thm:main} cannot rely on length alone.
It must exploit the Hall decomposition of each bracket. 
The main difficulty of this paper is to show that the decay of Sussmann's coefficients and the growth of the corresponding analytic Lie brackets compensate each other once the combinatorial structure of the Hall set is kept throughout the estimates (see in particular \cref{prop:Gammab,prop:Fb} and \cref{sec:univ-bound}, which contain the main arguments of this paper).

\subsection{Organization of the paper}

The paper is organized as follows.
\begin{itemize}
    \item In \cref{sec:Hall}, we recall Hall sets and prove estimates for Sussmann's coordinates $\xi_b$. 
    \item In \cref{sec:Lie}, we establish analytic estimates for the iterated Lie brackets~$f_b$. 
    \item The balance between the decaying coefficients and the growing brackets is encoded in \cref{sec:W} by a logarithmic sequence indexed by the Hall set, for which we prove a linear estimate.
    \item In \cref{sec:proof}, we combine these bounds to prove convergence of the infinite product of flows.
\end{itemize}

\section{Hall sets}
\label{sec:Hall}

We recall the definition and basic properties of Hall sets which will be used in the sequel.
We define and estimate the coordinates of the second kind $\xi_b$.

We follow Viennot's convention of \cite{Viennot1978} for Hall sets, but other conventions exist (e.g.\ in \cite{Reutenauer2003}).

\subsection{Definitions}

Let $X$ be a set.
Let $\Br(X)$ denote the free magma over $X$.
Elements of $\Br(X)$ can be represented as binary trees, e.g.\ $(X_0,(X_0,X_1))$, when $X_0, X_1 \in X$.
For $b \in \Br(X)$, $|b| \ge 1$ denotes its length, defined by $|b| = 1$ for $b \in X$, and $|(a,b)| = |a| + |b|$.

\begin{definition}
    \label{def:Hall}
    A \emph{Hall set} is a subset $\B$ of $\Br(X)$ endowed with a total order $<$ such that
    \begin{itemize}
        \item $X \subset \B$,
        \item for all $a, b \in \Br(X)$, $(a, b) \in \B$ iff $a, b \in \B$, $a < b$ and ($b \in X$ or $b = (b',b'')$ with $b' \le a$),
        \item for all $a, b \in \B$ such that $(a,b) \in \B$ then $a < (a,b)$.
    \end{itemize}
\end{definition}

Note that, in \cite{Sussmann1986}, Sussmann used a stronger definition of Hall sets, closer to the original works of \cite{Hall1950,Hall1933}, where one requires that the order is compatible with length (i.e.\ where $a < b \Rightarrow |a| \le |b|$).
This simplifies some arguments, but we proved in \cite{BeauchardLeBorgneMarbach2023} that his infinite product is still well-defined for non-length-compatible Hall sets, as suggested by \cite{Kawski1999}.

\begin{lemma} 
    \label{lem:facto}
    Let $\B$ be a Hall set on $X$.
	For any $b \in \B$, either $|b| = 1$ and $b \in X$, or $|b| > 1$ and there exists a unique tuple $(a,c,m) \in \B \times \B \times \N^*$ such that $b = \ad^m_a(c)$, $a < c$ and either $c \in X$ or $c = (c',c'')$ with $c' < a$.
    Since $|b| = m |a| + |c|$ with $m \ge 1$, one has $|a| < |b|$ and $|c| < |b|$.
\end{lemma}

\begin{proof}
	This is a direct consequence of \cref{def:Hall}, up to iterating the decomposition until the right factor is either in $X$ or satisfies the strict inequality $c' < a$ (not just $c' \le a$).
\end{proof}

The fact that length is strictly decreasing in \cref{lem:facto} allows to use such a decomposition as a well-founded induction to define functions on a Hall set $\B$ inductively.

This induction process is linked with Lazard elimination \cite{Lazard1960} and the following definition.

\begin{definition} 
    \label{def:Lazard}
    A \emph{Lazard set} is a subset $\B$ of $\Br(X)$, with a total order $<$ and such that, for every $L \in\N^*$, the set $\B_{\le L}=\{b_1 < \dotsb < b_{k+1}\}$ of elements of $\B$ with length at most $L$ satisfies
    \begin{equation} 
        \left\lbrace
        \begin{aligned}
            & b_1 \in Y_0:=X, \\
            & b_2 \in Y_1:=\{ \ad_{b_1}^m(v) ; m \in \N, v \in Y_0\setminus \{b_1\} \}, \\
            & \dots \\
            & b_{k+1} \in Y_k:=\{ \ad_{b_k}^m(v) ; m \in \N, v \in Y_{k-1} \setminus\{b_k\} \}
        \end{aligned}
        \right.
    \end{equation}
    and $\B_{\le L} \cap Y_k = \{ b_{k+1} \}$ (or, equivalently, $\B_{\le L} \cap Y_{k+1} = \emptyset$).
\end{definition}

One has the following equivalence \cite[Corollary 1.1]{Viennot1978}.

\begin{lemma}
    \label{lem:Hall=Lazard}
    A totally ordered subset $\B$ of $\Br(X)$ is a Hall set if and only if it is a Lazard set.
\end{lemma}

The main interest of Hall sets or Lazard sets is that they allow to construct bases of $\mathcal{L}(X)$.
Indeed, there is a natural map $\eval : \Br(X) \to \mathcal{L}(X)$ defined by induction by $\eval(x) := x$ for $x \in X$ and $\eval((a,b)) := [\eval(a),\eval(b)]$.
This leads to the following result \cite[Prop.\ 1.1 and Theorem 1.1]{Viennot1978}.

\begin{proposition}
    Let $\B$ be a Hall set on $X$.
    Then $\eval(\B)$ is a basis of $\mathcal{L}(X)$.
\end{proposition}

Different choices of orders on $\Br(X)$ lead to different Hall sets, and thus different bases of $\mathcal{L}(X)$, sometimes with very different properties (see for example \cite{BeauchardLeBorgneMarbach2022} where we give many examples of such orders, and study the structure constants of the associated bases).

\subsection{Coordinates of the second kind}

We recall the explicit formula defining coordinates of the second kind indexed by a Hall set $\B$ on a finite set $X$, due to Sussmann in \cite[equation (19)]{Sussmann1986}, and we prove an inductive bound.

\begin{definition}
    \label{def:xi}
    Let $\B$ be a Hall set on $X = \{ X_i \mid i \in I \}$. 
    Let $T > 0$ and $u \in L^1((0,T);\R^I)$.
    We define the associated coordinates of the second kind by induction by
    \begin{equation}
        \xi_{X_i}(t,u) := \int_0^t u_i \quad \text{for } i \in I
        \qquad \text{and} \qquad
        \xi_b(t,u) := \frac{1}{m!} \int_0^t \xi_a^m \dot{\xi}_c
    \end{equation}
    when $b = \ad_a^m(c)$ is given by the unique representation of \cref{lem:facto}.

    For each $b \in \B$, $\xi_b(\cdot,u)$ is absolutely continuous in time and $\dot{\xi}_b(\cdot,u)$ exists almost everywhere.
\end{definition}

These coordinates satisfy the following estimates.

\begin{proposition}
    \label{prop:Gammab}
	Let $T > 0$ and $u \in L^1((0,T);\R^I)$.
    For $t \in [0,T]$, define 
    \begin{equation}
        \label{eq:Xi}
        \Xi(t) := \norm{u}_{L^1(0,t)} := \int_0^t \sum_{i \in I} |u_i|.
    \end{equation}
	For any $b \in \B$, one has
    \begin{align}
        \label{eq:dotxib-bound}
        |\dot{\xi}_b| & \le e^{\Gamma_b} \Xi^{|b|-1} \dot{\Xi}
        \qquad \text{a.e.\ on } (0,T), \\
        \label{eq:xib-bound}
        |\xi_b| & \le \frac{e^{\Gamma_b}}{|b|} \Xi^{|b|}
        \qquad\quad \text{ on } [0,T],
    \end{align}
    where $\Gamma : \B \to \R$ is defined by setting $\Gamma_{X_i} := 0$ for $i \in I$ and
	\begin{equation} 
        \label{eq:def.gamma.b}
		\Gamma_b := m \Gamma_a + \Gamma_c - m \ln |a| - \ln(m!)
	\end{equation}
    when $b = \ad_a^m(c)$ is given by the unique representation of \cref{lem:facto}.
\end{proposition}

\begin{proof}
    For each fixed $b \in \B$, \eqref{eq:xib-bound} is obtained by integrating \eqref{eq:dotxib-bound} from the initial condition $\xi_b(0) = 0$.
    For $i \in I$, $\dot{\xi}_{X_i} = u_i$ by \cref{def:xi} so \eqref{eq:dotxib-bound} holds.
    By \cref{def:xi} for $b = \ad^m_a(c)$ and the induction hypothesis, one has
	\begin{equation}
		| \dot{\xi}_b | 
		\le \frac{|\xi_a|^m}{m!} | \dot{\xi}_c |
		\le \frac{1}{m!} \left(\frac{e^{\Gamma_a}}{|a|}\right)^m e^{\Gamma_c} \Xi^{m|a|+|c|-1} \dot{\Xi},
	\end{equation}
	which proves \eqref{eq:dotxib-bound} provided one sets $\Gamma_b$ as in \eqref{eq:def.gamma.b}.
\end{proof}

\section{Estimates of Lie brackets of analytic vector fields}
\label{sec:Lie}

In this section, we prove estimates on Lie brackets of real-analytic vector fields which we will need to prove the convergence of Sussmann's infinite product.
Here, $\Omega$ denotes an open subset of $\R^d$.

\subsection{Analytic norms and elementary estimates}
\label{sec:analytic-def}

Given $r > 0$, we define the space $C^{\omega,r}(\Omega;\R^q)$ as the subset of $C^\infty(\Omega;\R^q)$ for which the following norm is finite
\begin{equation} \label{eq:def.analytic.r}
    \norm{f}_r := \sum_{i = 1}^q \sum_{\alpha \in \N^d} \frac{r^{|\alpha|}}{\alpha!} \sup_{\Omega} |\partial^\alpha f_i|.
\end{equation}
The following standard properties of the norm $\norm{\cdot}_r$ are proved in \cite[Lemmas 70 and 71]{BeauchardLeBorgneMarbach2023}.

\begin{lemma}[Algebra property]
    \label{lem:alg-fg-r}
	For any $r > 0$ and $f, g \in C^{\omega,r}(\Omega;\R)$, one has
	\begin{equation}
		\norm{f g}_r \le \norm{f}_r \norm{g}_r.
	\end{equation}
\end{lemma}

\begin{lemma}[Gradient]          
    \label{lemma:gradient.estimate}
	For any $r>0$, $\rho>0$, $f \in C^{\omega,r+\rho}(\Omega;\R)$ and $1 \le j \le d$,
	\begin{equation}
		\norm{\partial_j f}_r \le \frac{1}{\rho} \norm{f}_{r+\rho}.
	\end{equation}
\end{lemma}

Given two vector fields $f, g \in C^\infty(\Omega;\R^d)$, we define their Lie bracket $[f,g] \in C^\infty(\Omega;\R^d)$ as
\begin{equation}
    [f,g] := Dg \cdot f - Df \cdot g = \sum_{j = 1}^d f_j \partial_j g - g_j \partial_j f.
\end{equation}
We also define the linear map $\ad_f : g \mapsto [f,g]$, for which we have the following estimate.

\begin{lemma}[Lie brackets]
    \label{lem:lie}
	For any $r>0$, $\rho>0$, $f, g \in C^{\omega,r+\rho}(\Omega;\R^d)$ and $m \in \N^*$,
	\begin{equation} 
        \label{eq:admfg.rdr}
        \norm{\ad^m_{f} (g)}_r \le \left(\frac{2m}{\rho}\right)^m \norm{f}_{r+\rho}^m \norm{g}_{r+\rho}.
	\end{equation}
\end{lemma}

\begin{proof}
	For $m = 1$, the estimate is a consequence of \cref{lem:alg-fg-r,lemma:gradient.estimate} and the choice of norm~\eqref{eq:def.analytic.r} which allows to sum over the dimensions.
    For $m > 1$, one applies the $m=1$ estimate $m$ times with a radius loss $\rho / m$ at each step. 
    This yields
	\begin{equation}
		\begin{split}
			\norm{\ad^m_{f} (g)}_r 
			& \le 
			\left(\frac{2m}{\rho}\right) \norm{\ad^{m-1}_f g }_{r+\frac{\rho}{m}}
			\norm{f}_{r+ \frac{\rho}{m}} \\
			& \le 
			\left(\frac{2m}{\rho}\right)^m \norm{g}_{r+\rho}
			\prod_{j=1}^m \norm{f}_{r+j \frac{\rho}{m}},
		\end{split}
	\end{equation}
	which concludes the proof because the norm \eqref{eq:def.analytic.r} is non-decreasing with respect to $r$.
\end{proof}

\subsection{Inductive estimates in a Hall set}

In this subsection $I$ is a finite set, and $\B$ is a Hall set over $X := \{ X_i \mid i \in I \}$.

\begin{definition}[Lie brackets $f_b$]
    \label{def:f_b}
    Given vector fields~$f_i \in C^\infty(\Omega;\R^d)$ for $i \in I$, we define their iterated Lie brackets $f_b$ for $b \in \Br(X)$ as follows. 
    By induction on $|b|$, we set $f_{X_i} := f_i$ for $i \in I$ and $f_{(a,b)} := [f_a, f_b]$ for $a,b \in \Br(X)$.
    Equivalently, for $b \in \Br(X)$, $f_b$ is the image of $\eval(b) \in \mathcal{L}(X)$ by the unique homomorphism of Lie algebras from $\mathcal{L}(X)$ to $C^\infty(\Omega;\R^d)$ mapping $X_i$ to $f_i$.
\end{definition}

\begin{proposition}
    \label{prop:Fb}
	Let $r > 0$, $\rho > 0$ and $f_i \in C^{\omega,r+\rho}(\Omega;\R^d)$ for $i \in I$.
	For any $b \in \B$, 
	\begin{equation} \label{eq:fb.Fb}
		\norm{f_b}_r
		\le e^{F_b}
		\left(\frac{2e}{\rho}\right)^{|b|-1}
		M_f^{|b|}
        \quad \text{where} \quad 
        M_f := \max_{i \in I} \norm{f_i}_{r+\rho}
	\end{equation}
	and $F : \B \to \R_+$ is defined by setting $F_b := 0$ for $b \in X$ and
	\begin{equation} \label{eq:def:Fb}
		F_b := m F_a + F_c + m \ln |b|
	\end{equation}
    when $b = \ad_a^m(c)$ is given by the unique representation of \cref{lem:facto}.
\end{proposition}

\begin{proof}
    We argue by induction on $|b|$. 
    The result is immediate for $b\in X$.
    Let now $b=\ad_a^m(c)$.
    For any $\sigma \in (0,1)$, by \cref{lem:lie} with radius loss $\sigma \rho$,
    \begin{equation}
        \label{eq:fbr-loss-sigma}
        \norm{f_b}_r \le \left(\frac{2m}{\sigma \rho}\right)^{m} \norm{f_a}_{r + \sigma\rho}^m \norm{f_c}_{r + \sigma\rho}.
    \end{equation}
    When $|b| = m+1$, $a, c \in X$ and \eqref{eq:fbr-loss-sigma} with $\sigma = 1/e$ entails \eqref{eq:fb.Fb} because $F_b = m \ln |b|$.
    
    We now assume that $|b| > m + 1$.
    By the induction hypothesis with radius loss $(1-\sigma)\rho$,
    \begin{equation}
        \norm{f_a}_{r+\sigma\rho} \le e^{F_a} \left(\frac{2e}{(1-\sigma)\rho}\right)^{|a|-1} M_f^{|a|}
        \quad \text{and} \quad 
        \norm{f_c}_{r+\sigma\rho} \le e^{F_c} \left(\frac{2e}{(1-\sigma)\rho}\right)^{|c|-1} M_f^{|c|}.
    \end{equation}
    Substituting in \eqref{eq:fbr-loss-sigma} leads to
    \begin{equation}
        \norm{f_b}_r \le e^{m F_a + F_c}  \left(\frac{m}{e\sigma}\right)^{m} \left(\frac{1}{1-\sigma}\right)^{|b|-1-m} \left(\frac{2e}{\rho}\right)^{|b|-1} M_f^{|b|}.
    \end{equation}
    Let $n := |b| - 1$.
    We choose $\sigma := m / n < 1$.
    One has\footnote{We use the classical inequality $\ln(1+x) \le x$, so $(1+x)^\alpha \le e^{\alpha x}$, so $(1+\frac{m}{n-m})^{n-m} \le e^m$.}
    \begin{equation}
        \left(\frac{m}{e\sigma}\right)^{m} \left(\frac{1}{1-\sigma}\right)^{|b|-1-m}
        =
        \left(\frac{n}{e}\right)^{m} \left(\frac{n}{n-m}\right)^{n-m}
        \le n^m \le e^{m \ln |b|},
    \end{equation}
    which concludes the proof recalling \eqref{eq:def:Fb}.
\end{proof}

\section{Estimate of the log-sequence}
\label{sec:W}

In this section, $X$ is a set and $\B$ a Hall set on $X$.
Motivated by \cref{prop:Gammab,prop:Fb}, we study the growth of the sequence $W_b := \Gamma_b + F_b$.
Our goal is to understand whether it grows at most linearly with respect to $|b|$.
Recalling \eqref{eq:def.gamma.b} and \eqref{eq:def:Fb}, we can equivalently define $W$ as follows.

\begin{definition}[Log-sequence]
    \label{def:W}
    We define $W : \B \to \R$ by setting $W_b := 0$ for $b \in X$ and
    \begin{equation}
        \label{eq:W}
        W_b := m W_a + W_c + m \ln \frac{|b|}{|a|} - \ln (m!)
    \end{equation}
    when $b = \ad_a^m(c)$ is given by the unique representation of \cref{lem:facto}.
\end{definition}

\begin{example}
    Let $X_0, X_1 \in X$ and assume that $X_0 < X_1$ in $\B$.
    Then $(X_0,X_1) \in \B$ and $W_{(X_0,X_1)} = \ln 2 - \ln 2 = 0$.
    The Hall axioms of \cref{def:Hall} do not impose any order between $X_1$ and $(X_0,X_1)$.
    If $X_1 < (X_0, X_1)$, then $(X_1, (X_0,X_1)) \in \B$ and $W_{(X_1, (X_0,X_1))} = \ln 3$.
    Otherwise, $((X_0, X_1), X_1) \in \B$ and $W_{((X_0,X_1),X_1)} = \ln \frac{3}{2}$.
\end{example}

\subsection{A length-based estimate is not enough}

A natural strategy to estimate the growth of $W$ is to propagate an estimate on its largest possible value for a given bracket length.
This leads to the following estimate, which is insufficient for our purpose, but which we give as an illustration of why such a strategy fails.

\begin{lemma}
    For all $b \in \B$, $W_b \le \ln (|b|!)$.
\end{lemma}

\begin{proof}
    From \cref{def:W}, one obtains that $W_b \le \Lambda_{|b|}$ where $\Lambda_1 := 0$ and, for $n \ge 2$,
    \begin{equation}
        \label{eq:An}
        \Lambda_n := \max \left\{ m \Lambda_a + \Lambda_c + m \ln \frac{n}{a} - \ln (m!) \mid a,c,m \in \N^*, \enskip m a + c = n \right\}.
    \end{equation}
    Let us prove by strong induction that $\Lambda_n = \ln(n!)$.

    Given $(a,c,m) \in \N^*$ such that $n = m a + c$, by the multinomial formula,
    \begin{equation}
        \label{eq:AN-1}
        ((a-1)!)^m c! \le (n-m)!
    \end{equation}
    Moreover, for $1 \le m \le n$,
    \begin{equation}
        \label{eq:AN-2}
        (n-m)! n^m \le n! m!
    \end{equation}
    Combining \eqref{eq:AN-1} and \eqref{eq:AN-2} entails that
    \begin{equation}
        (a!)^m c! \frac{n^m}{a^m m!} \le n!
    \end{equation}
    Thus $\Lambda_n \le \ln(n!)$.
    Moreover, equality is achieved in these estimates when $a = m = 1$ and $c = n-1$.
    Hence, for all $b \in \B$, $W_b \le \Lambda_{|b|} = \ln (|b|!)$.
\end{proof}

In fact, when $X$ is infinite, some Hall sets may saturate this bound, as we now illustrate.

\begin{lemma}
    Assume that $X$ is infinite.
    There exists a Hall set $\B$ on $X$ and a sequence $b_n \in \B$ such that, for all $n \ge 2$, $|b_n| = n$ and $W_{b_n} = \ln (n!)$.
\end{lemma}

\begin{proof}
    Since $X$ is infinite, let $X_0, X_1, X_2, \dotsc \in X$.
    Consider any order on $\Br(X)$ compatible with length (i.e.\ such that $|a| < |b| \Rightarrow a < b$) and such that $X_1 < X_0$ and $X_1 < X_2 < X_3 < \dotsb$.
    Let $\B$ be the unique Hall set associated with this order (see \cite[Lemma 1.37]{BeauchardLeBorgneMarbach2022}).
    Then, for all $n \ge 2$, the bracket $b_n := (X_{n-1},(X_{n-2}, \dotsc, (X_1, X_0) \dotsb)) \in \B$.
    Moreover, $|b_n| = n$ and $W_{b_n} = \ln (n!)$.
\end{proof}

\subsection{An easy sufficient condition for a particular class of Hall sets}
\label{sec:lyndon}

We identify a class of Hall sets for which it is straightforward to prove that $W$ grows at most linearly.
We show that it applies in particular to the Lyndon basis.

\begin{lemma}
    Let $\B$ be a Hall set on $X$. 
    Assume that there exists $M>0$ such that
    \begin{equation}
        \label{eq:c/a-bounded}
        \forall b \in \B \setminus X, \quad |c| \le M |a|,
    \end{equation}
    where $b = \ad_a^m(c)$ as in \cref{lem:facto}.
    Then, for all $b \in \B$, $W_b \le |b| (1 + \ln (1+M))$.
\end{lemma}

\begin{proof}
    Let $A := 1 + \ln(1+M)$ and $W'_b := W_b + A$.
    Then the induction \eqref{eq:W} and \eqref{eq:c/a-bounded} lead to
    \begin{equation}
        W_b' = m W_a' + W_c' + \left( m \ln \left(m+\frac{|c|}{|a|}\right) - \ln (m!) - m A \right) \le m W_a' + W_c'.
    \end{equation}
    Indeed, since $m! \ge (m/e)^m$, one has
    \begin{equation}
        m \ln \left(m+\frac{|c|}{|a|}\right) - \ln (m!)
        \le m \left(1 + \ln \left(1+ \frac{|c|}{m|a|}\right) \right) \le m A.
    \end{equation}
    By a strong induction on $|b|$, this entails that $W_b \le A |b|$ for all $b \in \B$.
\end{proof}

This holds for the Hall set obtained by standard factorization of Lyndon words (see \cite{ChenFoxLyndon1958,Lyndon1954,Shirshov1958}).
Let $X^*$ denote the set of words on $X$, and $X^+$ the non-empty ones.
Given an order on the alphabet~$X$, let $\prec$ be the associated lexicographic order on $X^*$.
A Lyndon word is an element $f \in X^+$ such that, for all $u, v \in X^+$, if $f = u v$, then $f \prec v u$.
Let $\Lyn$ be the set of Lyndon words.
Given $f \in \Lyn$, let $f'$ be its longest proper prefix in $\Lyn$.
A well-known result \cite{ChenFoxLyndon1958} states that, in such a case $f = f' f''$ where $f'' \in \Lyn$.
This allows us to define a \emph{standard factorization} map $\pi : \Lyn \to \Br(X)$ by setting $\pi(x) := x$ for $x \in X$ and $\pi(f) := (\pi(f'), \pi(f''))$.
Then $\pi(\Lyn)$ is a Hall set (see \cite[Theorem 5.1]{Reutenauer1993}).
We refer to \cite[Section 7.1]{BeauchardLeBorgneMarbach2022} for more details.

We prove the following lemma in \cref{sec:lyndon-proof}.

\begin{lemma}
    \label{lem:lyndon}
    Let $X$ be an ordered set and $\B$ be the associated Lyndon Hall set.
    For any $b \in \B \setminus X$, one has $|c| \le |a|$ where $b = \ad_a^m(c)$ as in \cref{lem:facto}.
\end{lemma}

In many Hall sets, there need not exist a finite constant $M$ such that \eqref{eq:c/a-bounded} holds.
As an example, given $X_0, X_1 \in X$ and $k \ge 2$, consider $b_k := ((X_0,X_1),\ad_{X_0}^k(X_1))$.
For this bracket, $|a| = 2$ and $|c| = k+1$ so $|c|/|a|$ can be arbitrarily large.
It belongs to the following Hall sets:
\begin{itemize}
    \item the historical length-compatible bases of \cite{Hall1950,Hall1933} when $X_0 < X_1$;
    
    \item the left-right lexicographic basis of \cite[Section 8.2]{BeauchardLeBorgneMarbach2022};
    
    \item the $\B^\star$ basis constructed for control theory in \cite[Section 3]{BeauchardMarbach2026}, up to exchanging $X_0$ and $X_1$.
\end{itemize}

\subsection{A universal linear bound}
\label{sec:univ-bound}

In this section, $X$ is finite and nonempty and $\B$ is an arbitrary Hall set over $X$. 

To control the growth of $W_b$, we will perform estimates on a kind of ``partition function'' depending on a regularity parameter $\beta \ge 0$. 
For $Z \subset \B$, let
\begin{equation}
    \label{eq:def-Pp}
    \mathrm{P}_Z(\beta) := \sum_{b \in Z} p_b(\beta)
    \quad \text{where} \quad 
    p_b(\beta) := \frac{1}{|b|} \exp \big(W_b - \beta |b|\big).
\end{equation}
The natural induction strategy within a Hall set is called ``Lazard elimination''.
Let $a, c \in \B$ such that $a < c$ and ($c \in X$ or $c = (c',c'')$ with $c' < a$).
We define the set 
\begin{equation}
    \label{eq:def-a*c}
    a^* c := \{ \ad_a^m(c) : m \in \N \} \subset \B.  
\end{equation}
Iterating such eliminations exhausts the full Hall set inductively.
The key idea of this section is that, up to small loss in the regularity parameter $\beta$, one can control $\mathrm{P}_{a^*c}$ from $p_c$.

\begin{proposition}
    \label{prop:one-step-lazard-contraction}
    Let $a < c \in \B$ such that $c \in X$ or $c = (c',c'')$ with $c' < a$.
    For any $\beta \ge 0$,
    \begin{equation}
        \label{eq:one-step-lazard-contraction}
        \mathrm{P}_{a^* c}(\beta+p_a(\beta)) \le p_c(\beta).
    \end{equation}
\end{proposition}

\begin{proof}
    Let $t := |c| / |a|$.
    By the inductive definition \eqref{eq:W},
    \begin{equation}
        e^{W_{\ad_a^m(c)}} = \frac{(m+t)^m}{m!} e^{m W_a} e^{W_c}.
    \end{equation}
    Thus, by \eqref{eq:def-Pp} and \eqref{eq:def-a*c},
    \begin{equation}
        \begin{split}
            \mathrm{P}_{a^*c}(\beta + p_a(\beta))
            & = \sum_{m \ge 0} \frac{1}{m|a|+|c|} e^{W_{\ad_a^m(c)}} e^{-(\beta + p_a(\beta)) (m|a|+|c|)} \\
            & = \sum_{m \ge 0} 
            \frac{|c|}{m|a|+|c|} \frac{(m+t)^m}{m!} \left(e^{W_a - \beta|a| - |a| p_a(\beta)}\right)^m \frac{e^{W_c - \beta|c|}}{|c|} e^{- |c| p_a(\beta)} \\
            & = K_{\alpha}(t) p_c(\beta) e^{-\alpha t}
        \end{split}
    \end{equation}
    where $\alpha := |a| p_a(\beta)$ and $K_\alpha(t)$ is defined in \eqref{eq:k-alpha} below.
    By \cref{lem:K_alpha}, $K_\alpha(t) \leq e^{\alpha t}$, which concludes the proof of \eqref{eq:one-step-lazard-contraction}.
\end{proof}

\begin{lemma}
    \label{lem:K_alpha}
    For any $\alpha, t > 0$,
    \begin{equation}
        \label{eq:k-alpha}
        K_\alpha(t) := \sum_{m\ge0} \frac{t}{m+t} \frac{(m+t)^m}{m!} (\alpha e^{-\alpha})^m \le e^{\alpha t}.
    \end{equation}
\end{lemma}

\begin{proof}
    For $z \in [0,\frac 1e]$, let $A(z)$ denote the smallest nonnegative solution of $A(z) = z e^{A(z)}$ (one has $A(z) = - W_0(-z)$ where $W_0$ is the principal branch of the Lambert function).

    By Lagrange inversion (see \cite[Theorem 1]{SuryaWarnke2023}, with $\Phi = \exp$ and $H = \exp(t\cdot)$; or \cite[eq.\ (3.2.3)]{Gessel2016}), one has:
    \begin{equation}
        e^{t A(z)} = \sum_{m\ge0}
        \frac{t}{m+t}
        \frac{(m+t)^m}{m!}z^m.
    \end{equation}
    Hence $K_\alpha(t) = e^{t A(\alpha e^{-\alpha})}$.
    Since $A(\alpha e^{-\alpha}) \leq \alpha$, the conclusion follows.
\end{proof}

In \cref{prop:one-step-lazard-contraction}, the regularity loss is precisely controlled by $p_a(\beta)$.
This allows for a bootstrap argument, where the energy controls the loss, which controls the energy.

\begin{proposition}
    \label{prop:global-pressure-bound}
    Let $X$ be finite, and let $\B$ be a Hall set over $X$. 
    For any $\beta \ge 0$,
    \begin{equation}
        \label{eq:global-pressure-bound}
        \mathrm{P}_\B(\beta + |X| e^{-\beta}) \le |X| e^{-\beta}.
    \end{equation}
\end{proposition}

\begin{proof}
    By \cref{lem:Hall=Lazard}, $\B$ is a Lazard set.
    Fix $L\ge1$.
    Write $\B_{\le L} = \{ b_1 < \dotsb < b_{k+1} \}$.
    Define $\beta_0 := \beta$ and, for $1 \le i \le k+1$,
    \begin{equation}
        \beta_i := \beta_{i-1} + p_{b_i}(\beta_{i-1}).
    \end{equation}
    With the notations of \cref{def:Lazard}, $Y_i = \bigcup_{c \in Y_{i-1} \setminus \{ b_i \}} b_i^* c$.
    By \cref{prop:one-step-lazard-contraction},
    \begin{equation}
        \begin{split}
            \beta_i + \mathrm{P}_{Y_i}(\beta_i) 
            & = \beta_{i-1} + p_{b_i}(\beta_{i-1}) 
            + \sum_{c \in Y_{i-1} \setminus \{ b_i \}} \mathrm{P}_{b_i^* c}\left( \beta_{i-1} + p_{b_i}(\beta_{i-1}) \right) \\
            & \le \beta_{i-1} + p_{b_i}(\beta_{i-1})
            + \sum_{c \in Y_{i-1} \setminus \{ b_i \}} p_c(\beta_{i-1}) \\ 
            & = \beta_{i-1} + \mathrm{P}_{Y_{i-1}}(\beta_{i-1}).
        \end{split}
    \end{equation}
    Hence, for $0 \le i \le k+1$, we have $\beta_i \le \beta + \mathrm{P}_{Y_0}(\beta) = \beta + |X| e^{-\beta} =: \gamma$.
    Thus,
    \begin{equation}
        \mathrm{P}_{\B_{\le L}}(\gamma) 
        = \sum_{i=1}^{k+1} p_{b_i}(\gamma)
        \le \sum_{i=1}^{k+1} p_{b_i}(\beta_{i-1})
        = \beta_{k+1} - \beta \le |X| e^{-\beta}.
    \end{equation}
    Letting $L\to+\infty$ proves \eqref{eq:global-pressure-bound}.
\end{proof}

\begin{corollary}
    \label{cor:linear-bound}
    Let $X$ be finite, and let $\B$ be a Hall set over $X$. 
    There exists $C > 0$ such that, for any $b \in \B$,
    \begin{equation}
        \label{eq:bound-Wb}
        W_b \le C |b|.
    \end{equation}
\end{corollary}

\begin{proof}
    Applying \cref{prop:global-pressure-bound} with $\beta = \ln |X|$ proves that $\mathrm{P}_\B(1+\ln |X|) \le 1$.
    In particular, for any $b\in\B$, one has $p_b(1+\ln |X|) \le 1$, which entails \eqref{eq:bound-Wb} with $C := 2 + \ln |X|$ using $\ln |b| \le |b|$.
\end{proof}

\section{Convergence of Sussmann's infinite product}
\label{sec:proof}

Combining the arguments of the previous sections, we prove \cref{thm:main}.

In this section, for $\delta > 0$, $B_\delta$ denotes the open ball in $\R^d$ of radius $\delta$ centered at $0$.

\subsection{Infinite products of flows}
\label{sec:proof-def}

To define infinite products of flows of vector fields, we use the following notions.

\begin{definition}
    Let $\Omega$ be an open subset of $\R^d$, $J = \{ j_1 < \dotsb < j_n \}$ a finite ordered set and $(f_j)_{j \in J} \in (C^1(\Omega;\R^d))^{J}$ a family of vector fields.
    When it makes sense, we define
    \begin{equation}
        \left(\overset{\longrightarrow}{\prod_{j \in J}} e^{f_j}\right)(p)
        := e^{f_{j_1}} \circ \dotsb \circ e^{f_{j_n}} (p). 
    \end{equation}
\end{definition}

\begin{definition}
    Let $\Omega$ be an open subset of $\R^d$, $J$ an ordered set and $(f_j)_{j \in J} \in (C^1(\Omega;\R^d))^{J}$ a family of vector fields.
    Given an open subset $\Omega' \subset \Omega$ and $g \in C^0(\Omega';\R^d)$, we say that the ordered product of the $e^{f_j}$ over $J$ converges towards $g$ on $\Omega'$ when, for any $\varepsilon > 0$, there exists a finite subset $J_\varepsilon \subset J$ such that, for any finite subset $J_\varepsilon \subset J' \subset J$, and any $p \in \Omega'$, one has
    \begin{equation}
        \left| g(p) - \left(\overset{\longrightarrow}{\prod_{j \in J'}} e^{f_j}\right)(p) \right| \le \varepsilon.
    \end{equation}
    When such a $g$ exists, it is unique, and we write
    \begin{equation}
        g = \overset{\longrightarrow}{\prod_{j \in J}} e^{f_j}.
    \end{equation}
\end{definition}

One has the following sufficient condition for convergence (see \cite[Lemma 134]{BeauchardLeBorgneMarbach2023}).

\begin{lemma}
    \label{lem:crit-prod-conv}
    Let $J$ be an ordered set, $\delta > 0$ and $(f_j)_{j \in J} \in (C^1(B_{2\delta};\R^d))^J$.
    Assume that
    \begin{equation}
        \sum_{j \in J} \norm{f_j}_{C^0(B_{2\delta};\R^d)} \le \delta
        \quad \text{and} \quad
        \sum_{j \in J} \norm{f_j}_{C^1(B_{2\delta};\R^d)} < + \infty.
    \end{equation}
    Then the ordered product of the $e^{f_j}$ over $J$ converges on $B_\delta$.
\end{lemma}

\subsection{Useful lemmas}

The following conjugation result for real-analytic vector fields is proved in \cite[Lemma 90, item 3]{BeauchardLeBorgneMarbach2023}.

\begin{lemma}
    \label{lem:conj}
    Let $\delta, r > 0$ and $f, g \in C^{\omega,3r}(B_{3\delta};\R^d)$ such that $\norm{f}_{C^0} < \delta$ and $\norm{f}_{3r} < \frac{r}{6}$.
    Then
    \begin{equation}
        \bigl((e^{-f})_* g\bigr) (p)
        = \sum_{k=0}^{+\infty} \frac{1}{k!} \ad_f^k(g)(p),
    \end{equation}
    for all $p \in B_{2\delta}$, where the series converges in $C^{\omega,2r}(B_{3\delta};\R^d)$.
\end{lemma}

\subsection{Proof of the main theorem}

We now prove the main theorem.
The proof scheme is inspired by \cite{Sussmann1985} and \cite[Section 4.5]{BeauchardLeBorgneMarbach2023}, taking advantage of the global pressure bound of \cref{sec:univ-bound}.

\begin{proof}[Proof of \cref{thm:main}]
    Since $I$ is finite and each $f_i$ is real-analytic on a neighborhood of $0$, there exist $r,\delta \in (0,1]$ such that $f_i \in C^{\omega,4r}(B_{3\delta};\R^d)$ for all $i \in I$.
    Set $M_f := \max_{i \in I} \norm{f_i}_{4r}$ and $\gamma := 1 + \ln |X|$.
    Fix $\eta>0$ small enough so that
    \begin{equation}
        \label{eq:eta-smallness}
        (1+M_f)\eta<\delta
        \quad\text{and}\quad
        \frac{2eM_f}{r}\eta
        <e^{-\gamma}\min\left(\delta,\frac r6\right).
    \end{equation}
    From now on, assume that $p \in \R^d$ and $u \in L^1((0,T);\R^I)$ are such that $|p| \leq \eta$ and $\norm{u}_{L^1} \leq \eta$. 

    \medskip
    \noindent
    \emph{Step 1: We prove that the infinite product is well-defined.}
    By \cref{prop:global-pressure-bound} with $\beta = \ln |X|$,
    \begin{equation}
        \label{eq:Pgamma}
        \mathcal{P}_\gamma
        :=
        \sum_{b\in\B}
        \frac1{|b|}
        \exp\bigl(W_b-\gamma|b|\bigr) \le 1.
    \end{equation}
    Let $\Xi$ be as in \eqref{eq:Xi}.
    In particular, $\Xi \le \norm{u}_{L^1} \le \eta$.
    Combining the coordinate estimate of \cref{prop:Gammab} with the Lie-bracket estimate of \cref{prop:Fb}, for all $b\in\B$ and $t\in[0,T]$,
    \begin{equation}
        \label{eq:xib-fb-Wb}
        \norm{\xi_b(t,u)f_b}_{3r}
        \le 
        \frac{e^{\Gamma_b}}{|b|}\Xi(t)^{|b|}
        e^{F_b}
        \left(\frac{2e}{r}\right)^{|b|-1}
        M_f^{|b|} 
        \le 
        \frac1{|b|}
        e^{W_b}
        \left(\frac{2eM_f}{r} \eta\right)^{|b|}.
    \end{equation}
    Combining \eqref{eq:Pgamma} and \eqref{eq:eta-smallness}, we obtain
    \begin{equation}
        \sum_{b\in\B}\norm{\xi_b(t,u)f_b}_{3r}
        < \min \left(\delta,\frac r6\right).
    \end{equation}
    Recalling \eqref{eq:def.analytic.r}, we thus have
    \begin{equation}
        \label{eq:xib-fb-C1-sum}
        \sum_{b\in\B}\norm{\xi_b(t,u)f_b}_{C^0(B_{3\delta})} \le \delta 
        \quad \text{and} \quad 
        \sum_{b\in\B}\norm{\xi_b(t,u)f_b}_{C^1(B_{3\delta})} \le \frac{\delta}{r} < +\infty. 
    \end{equation}
    Hence, for any $t \in [0,T]$, the infinite product of \eqref{eq:thm.main} converges on $B_{\delta}$ by \cref{lem:crit-prod-conv}.

    We shall also use a similar estimate for the differentiated coordinates.
    The derivative estimate in \cref{prop:Gammab} and the Lie-bracket estimate in \cref{prop:Fb} give
    \begin{equation}
        \label{eq:bound-dotxi-fb}
        \int_0^t
        |\dot\xi_b|
        \|f_b\|_{3r}
        \dd s
        \le 
        e^{\Gamma_b+F_b}
        \left(\frac{2e}{r}\right)^{|b|-1}
        M_f^{|b|}
        \int_0^t
        \Xi^{|b|-1}\dot\Xi =
        \frac{r}{2e}
        \frac1{|b|}
        e^{W_b} \left(\frac{2eM_f}{r} \eta\right)^{|b|}.
    \end{equation}
    The right-hand side is summable over $b\in\B$ by \eqref{eq:Pgamma} and \eqref{eq:eta-smallness}.

    \medskip
    \noindent
    \emph{Step 2: We prove that the trajectory is defined on $[0,T]$ and remains in $B_\delta$.}
    Indeed, as long as the state remains in $B_{2\delta}$, using \eqref{eq:eta-smallness},
    \begin{equation}
        |x(t;u,p)| 
        \le |p| + \int_0^t \sum_{i \in I} |u_i| \norm{f_i}_{C^0(B_\delta)}
        \le |p| + M_f \norm{u}_{L^1} \le \delta.
    \end{equation}
    The usual continuation argument shows that $x(t;u,p)$ is defined on $[0,T]$ and remains in $B_\delta$.

    \medskip
    \noindent
    \emph{Step 3: We reduce the result to the convergence of bounded length products.}
    Let $G(t)$ denote the infinite product of the right-hand side of \eqref{eq:thm.main}.
    For $L \ge 1$, let $G_L(t)$ be the finite ordered product over~$\B_{\le L}$.
    Since the finite sets $\B_{\le L}$ are cofinal among finite subsets of~$\B$, the convergence of $G(t)$ already proved implies, for all $t \in [0,T]$, $G_L(t) \to G(t)$ in $C^0(B_\delta;\R^d)$ as $L \to +\infty$.
    Hence it suffices to prove that $G_L(t)(p) \to x(t;u,p)$.
    
    \medskip
    \noindent
    \emph{Step 4: We construct auxiliary systems by Lazard elimination.}
    Let $L \ge 1$.
    By \cref{lem:Hall=Lazard}, $\B$ is a Lazard set.
    Using the notations of \cref{def:Lazard}, write $\B_{\le L} = \{ b_1 < \dotsb < b_{k+1} \}$ and let $Y_0, \dotsc, Y_{k+1}$ be the associated Lazard elimination sets.
    We use the standard property of Lazard elimination for Hall sets that each $Y_j$ is a subset of $\B$.
    In particular, $Y_{k+1}\subset\B$ and $\B_{\le L}\cap Y_{k+1}=\emptyset$, hence
    $Y_{k+1}\subset \{b\in\B: |b|>L\}$.

    Define successively
    \begin{equation}
        z_0(s):=x(s;u,p),
    \end{equation}
    and, for $1\le j\le k+1$,
    \begin{equation}
        \label{eq:zj-def}
        z_j(s)
        :=
        e^{-\xi_{b_j}(s,u)f_{b_j}}
        \bigl(z_{j-1}(s)\bigr).
    \end{equation}
    Since $x(s;u,p)\in B_\delta$ and $\sum_{b\in\B}\norm{\xi_b(s,u)f_b}_{C^0(B_{3\delta})}\le \delta$, all these curves remain in $B_{2\delta}$.
    Indeed, during each inverse flow the displacement is bounded by the $C^0(B_{3\delta})$ norm of the corresponding vector field, and the cumulative displacement is controlled by \eqref{eq:xib-fb-C1-sum}.
    In particular, the intermediate flow trajectories remain in $B_{3\delta}$, so the following computations are legitimate.

    We claim that, for every $0\le j\le k+1$,
    \begin{equation}
        \label{eq:zj-ODE}
        \dot z_j(s)
        =
        \sum_{b\in Y_j}
        \dot\xi_b(s,u)f_b(z_j(s)),
        \quad \text{and} \quad
        z_j(0)=p,
    \end{equation}
    for a.e.\ $s\in(0,T)$, where the series is absolutely convergent in $L^1((0,T);C^1(B_{2\delta}))$.

    For $j=0$, this is exactly the original differential equation \eqref{eq:syst}, because $Y_0=X$ and $\dot{\xi}_{X_i} = u_i$.
    Assume \eqref{eq:zj-ODE} holds at step $j-1$. 
    From \eqref{eq:zj-def}, the chain rule gives, for a.e.\ $s\in(0,T)$,
    \begin{equation}
        \label{eq:zj-ODE-conj}
        \dot z_j(s)
        =
        -\dot\xi_{b_j}(s,u)f_{b_j}(z_j(s)) +
        \sum_{v\in Y_{j-1}}
        \dot\xi_v(s,u)
        \left(
        e^{-\xi_{b_j}(s,u)f_{b_j}}
        \right)_*f_v(z_j(s)).
    \end{equation}
    The term $v=b_j$ cancels the first term, because the flow of $f_{b_j}$ preserves $f_{b_j}$.
    Moreover, the bounds above ensure that $\norm{\xi_{b_j}(s,u) f_{b_j}}_{C^0} < \delta$ and $\norm{\xi_{b_j}(s,u) f_{b_j}}_{3r} < r/6$.
    Thus, by \cref{lem:conj}, for every $v\in Y_{j-1}\setminus\{b_j\}$,
    \begin{equation}
        \label{eq:conj-xi}
        \left(
        e^{-\xi_{b_j}(s,u)f_{b_j}}
        \right)_*f_v
        =
        \sum_{m \ge 0}
        \frac{\xi_{b_j}(s,u)^m}{m!}
        f_{\ad_{b_j}^m(v)}
    \end{equation}
    with convergence in $C^{\omega,2r}$. 
    Substituting \eqref{eq:conj-xi} in \eqref{eq:zj-ODE-conj} and recalling \cref{def:xi} and the definition of $Y_j$ proves \eqref{eq:zj-ODE}.
    The absolute convergence required in the claim follows from \eqref{eq:bound-dotxi-fb}, since each $Y_j$ is contained in $\B$.

    \medskip
    \noindent
    \emph{Step 5: We prove the claimed convergence.}
    We now apply \eqref{eq:zj-ODE} at the final step $j=k+1$.
    Since $z_{k+1}(0)=p$, \eqref{eq:zj-ODE}, the inclusion $\B_{\le L} \cap Y_{k+1} = \emptyset$, and \eqref{eq:bound-dotxi-fb} yield
    \begin{equation}
        \label{eq:zk-p}
            |z_{k+1}(t)-p|
            \le 
            \int_0^t
            \sum_{b\in Y_{k+1}}
            |\dot\xi_b|\|f_b\|_{C^0(B_{2\delta})}
            \le 
            \frac{r}{2e}
            \sum_{\substack{b\in\B\\ |b|>L}}
            \frac1{|b|}
            e^{W_b}
            \left(\frac{2eM_f}{r}\eta\right)^{|b|}.
    \end{equation}
    The last expression tends to $0$ as $L\to+\infty$, as the tail of a convergent positive series.

    Finally, the finite product $G_L(t)$ has a Lipschitz constant bounded independently of $L$. 
    Indeed, by \eqref{eq:xib-fb-C1-sum},
    \begin{equation}
        \label{eq:Lip-GL}
        \operatorname{Lip}(G_L(t))
        \le 
        \exp\left(
        \sum_{b\in\B_{\le L}}
        \|\xi_b(t,u)f_b\|_{C^1(B_{3\delta})}
        \right)
        \le 
        \exp\left(\frac{\delta}{r}\right).
    \end{equation}
    By construction, $z_{k+1}(t)$ is obtained by applying the inverse finite product to $x(t;u,p)$. Hence
    \begin{equation}
        x(t;u,p)
        =
        G_L(t)\bigl(z_{k+1}(t)\bigr).
    \end{equation}
    Therefore, by \eqref{eq:zk-p} and \eqref{eq:Lip-GL},
    \begin{equation}
        \begin{aligned}
            |x(t;u,p)-G_L(t)(p)|
            =
            |G_L(t)(z_{k+1}(t))-G_L(t)(p)|
            \le 
            \exp\left(\frac{\delta}{r}\right)|z_{k+1}(t)-p|
        \end{aligned}
    \end{equation}
    This proves that $G_L(t)(p) \to x(t;u,p)$, which concludes the proof of \eqref{eq:thm.main}.
\end{proof}

\appendix

\section{Proof of the Lyndon length lemma}
\label{sec:lyndon-proof}

In this section, we prove \cref{lem:lyndon} of \cref{sec:lyndon}.
Let $\varepsilon$ denote the empty word.

We shall use the following standard facts about Lyndon words.
First, a nonempty word $f$ is Lyndon if and only if it is strictly smaller than each of its nonempty proper suffixes.
Second, if $u,v\in\Lyn$ and $u\prec v$, then $uv\in\Lyn$.

\begin{lemma}
    \label{lem:lyndon_prefix}
    Let $v, w \in X^*$ and $x \prec z \in X$.
    If $\ell = v x w \in \Lyn$, then $v z \in \Lyn$.
\end{lemma}

\begin{proof}
    If $v = \varepsilon$, then $v z = z \in \Lyn$.
    We may therefore assume that $v \neq \varepsilon$.

    Let $s = v_2 z$ with $v_2 \in X^*$ be a nonempty proper suffix of $v z$, and write $v = v_1 v_2$ with $v_1\in X^+$. 

    Since $v_2 x w$ is a nonempty proper suffix of $\ell$ and $x \prec z$, we have $\ell \prec v_2 x w \prec v_2 z = s$.

    Since $v$ is a proper prefix of $\ell$, we have $v\prec \ell$, and therefore $v\prec s$.

    Suppose, for contradiction, that $v z\succeq s$.
    Then $v \prec s \preceq v z$.
    Since $v z$ is obtained from $v$ by appending the single letter $z$, this forces $v$ to be a strict prefix of $s$.
    Hence $|v|<|s|$. 
    But $|v|=|v_1|+|v_2|\ge |v_2|+1=|s|$, a contradiction.

    Thus $v z \prec s$ for every nonempty proper suffix $s$ of $v z$, and therefore $v z \in\Lyn$.
\end{proof}

\begin{proposition}
    \label{prop:lyndon_description}
    Let $\ell \in \Lyn$, and let $\ell_1$ be its longest proper Lyndon prefix.
    Write $\ell = \ell_1\ell_2$.
    Assume that $\ell_1$ is not a prefix of $\ell_2$.
    Then $\ell_2 = \bar{\ell}_1 z$ for some possibly empty strict prefix $\bar{\ell}_1 \in X^*$ of $\ell_1$ and some letter $z \in X$.
\end{proposition}

\begin{proof}
    Recall that $\ell_2\in\Lyn$, and that $\ell_1 \prec \ell_2$.
    Let $p$ be the shortest prefix of $\ell_2$ such that $\ell_1 \prec p$.
    Such a prefix exists because $\ell_1 \prec \ell_2$.

    Since $\ell_1$ is not a prefix of $\ell_2$, the comparison $\ell_1 \prec p$ occurs at the first letter where $\ell_1$ and $\ell_2$ differ.
    Hence there exist a possibly empty strict prefix $\bar{\ell}_1 \in X^*$ of $\ell_1$, letters $x, z \in X$, and a word $r \in X^*$, such that $\ell_1 = \bar{\ell}_1 x r$, $p = \bar{\ell}_1 z$, $x \prec z$.
    By \cref{lem:lyndon_prefix}, $p = \bar{\ell}_1 z \in \Lyn$.
    
    If $p$ were a proper prefix of $\ell_2$, then $\ell_1 p$ would be a proper prefix of $\ell$.
    Moreover, since $\ell_1,p\in\Lyn$ and $\ell_1\prec p$, the word $\ell_1 p$ would be Lyndon.
    This contradicts the maximality of $\ell_1$ as the longest proper Lyndon prefix of $\ell$.
    Therefore $p=\ell_2$, and the claim follows.
\end{proof}

The following result can be deduced from \cite[Lemma 7.14]{Reutenauer1993}, the formulation given here is better suited for our purpose of proving \cref{lem:lyndon}.

\begin{corollary}
    \label{cor:lyndon_power_prefix}
    Let $\ell \in \Lyn$, and let $\ell_1$ be its longest proper Lyndon prefix.
    Then there exist an $m \geq 1$, a possibly empty strict prefix $\bar{\ell}_1 \in X^*$ of $\ell_1$, and a letter $z \in X$ such that $\ell = \ell_1^m \bar{\ell}_1 z$.
\end{corollary}

\begin{proof}
    Write the standard factorization as $\ell=\ell_1L_1$, with $L_1\in\Lyn$.
    
    We shall use the following observation. 
    Let $W\in\Lyn$ have longest proper Lyndon prefix $\ell_1$, and write $W=\ell_1R$ with $R\in\Lyn$.
    If $R$ has $\ell_1$ as a prefix, then $\ell_1$ is the longest proper Lyndon prefix of $R$. 
    Indeed, any proper Lyndon prefix $q$ of $R$ with $|q|>|\ell_1|$ has $\ell_1$ as a strict prefix, hence $\ell_1\prec q$; therefore $\ell_1q\in\Lyn$, contradicting the maximality of $\ell_1$ in $W$.
    
    Starting with $L_0=\ell$, repeat the standard factorization $L_{i-1} = \ell_1 L_i$ as long as $L_i$ has $\ell_1$ as a prefix.
    The observation shows that the hypothesis that $\ell_1$ is the longest proper Lyndon prefix is preserved at each step.
    Since lengths strictly decrease, the process terminates.
    Thus, for some $m\geq 1$, $\ell = \ell_1^m L_m$ where $L_m\in\Lyn$, $L_m$ does not have $\ell_1$ as a prefix, and $\ell_1$ is the longest proper Lyndon prefix of $L_{m-1}=\ell_1L_m$.
    
    Applying \cref{prop:lyndon_description} to $L_{m-1}$ gives $L_m = \bar{\ell}_1 z$ for some possibly empty strict prefix $\bar{\ell}_1$ of $\ell_1$ and some $z\in X$. 
    Hence $\ell=\ell_1^m\bar{\ell}_1z$, as claimed.
\end{proof}

\begin{proof}[Proof of \cref{lem:lyndon}]
    Let $b\in\B\setminus X$, and let $\ell\in\Lyn$ be such that $b=\pi(\ell)$.
    Let $\ell_1$ be the longest proper Lyndon prefix of $\ell$.
    By \cref{cor:lyndon_power_prefix}, there exist $m\geq 1$, a possibly empty strict prefix $\bar{\ell}_1\in X^*$ of $\ell_1$, and a letter $z\in X$ such that $\ell = \ell_1^m \bar{\ell}_1 z$.
    By the definition of the Lyndon Hall basis, this corresponds to the decomposition $b = \ad_a^m(c)$, where $a=\pi(\ell_1)$ and $c=\pi(\bar{\ell}_1z)$.
    Therefore
    \begin{equation}
        |c| = |\bar{\ell}_1z| = |\bar{\ell}_1|+1 \leq |\ell_1| = |a|,
    \end{equation}
    because $\bar{\ell}_1$ is a strict prefix of $\ell_1$.
\end{proof}

\section*{Acknowledgments}

We warmly thank Karine Beauchard for many discussions on this problem.

The authors acknowledge support from ANR-11-LABX-0020 (Labex Lebesgue), as well as from the Fondation Simone et Cino Del Duca -- Institut de France.

The proofs of \cref{sec:univ-bound} have been obtained through discussions with generative AI models.

\bibliographystyle{plain}
\bibliography{control}

\end{document}